\documentclass[12pt]{amsart}

\usepackage{amssymb,amsthm,amsfonts,latexsym}
\usepackage{amsmath}
\usepackage{color}
\usepackage{enumitem}
\usepackage{pxfonts} 
\usepackage{tikz-cd}
\usepackage{hyperref}

\allowdisplaybreaks

\theoremstyle{plain}
\newtheorem{theorem}{Theorem}[section]
\newtheorem*{theorem*}{Theorem}
\newtheorem{lemma}[theorem]{Lemma}

\newtheorem*{corollary*}{Corollary}

\newtheorem{mainthm}{Theorem}

\theoremstyle{definition}

\newtheorem{definition}[theorem]{Definition}

\theoremstyle{remark}

\numberwithin{equation}{section}

 \def\QQ{\mathbb{Q}}

 \def\D{{\mathcal D}}
 
 \def\F{{\mathcal F}}

 \def\N{{\mathcal N}}
 
 \def\P{{\mathcal P}}
 \def\Q{{\mathcal Q}}
 \def\R{{\mathcal R}}

 \def\U{{\mathcal U}}

 \def\Lk{{\text{Lk}}}
 \def\join{*}
   
 \def\ujoin{\,\underline{\join}\,}

 \newcommand{\GF}[1]{\mathbb{F}_{#1}}

 \DeclareMathOperator\PSL{PSL}

 \DeclareMathOperator{\PSU}{PSU}
 
 \DeclareMathOperator{\PSp}{PSp}

 \DeclareMathOperator{\POm}{P\Omega}
 
 \DeclareMathOperator\Sz{Sz}

 \DeclareMathOperator\Aut{Aut}
 
 \DeclareMathOperator\Out{Out}
 \DeclareMathOperator\Inn{Inn}
 
 \DeclareMathOperator\Inndiag{Inn diag}
 
 \newcommand{\tq}{\mathrel{{\ensuremath{\: : \: }}}}

 \def\Sp{\mathcal{S}_p} 
 \def\Ap{\mathcal{A}_p}
 \def\Bp{\mathcal{B}_p}

 \def\QD{(\mathcal{Q}\mathcal{D})}

 \newcommand\gen[1]{\left\langle#1\right\rangle}

\renewcommand{\gen}[1]{\langle #1 \rangle}
\def\Frob{\mathrm{Fr}}
\def\GG{\mathbb{G}}

\newcommand{\pushout}[4]{
\begin{tikzcd}[ampersand replacement=\&]
{\displaystyle\coprod_{#2}} \Lk_{#3}(#1) \arrow[d, hook] \arrow[r] \& #3 \arrow[d, hook] \\
{\displaystyle\coprod_{#2}} \Lk_{#3}(#1) * \{ #1\} \arrow[r] \& #4
\end{tikzcd}}

\title{Components in characteristic $p$ and Quillen's conjecture}
\author{Kevin Ivan Piterman}
\address{Vrije Universiteit Brussel \\
   Department of Mathematics \& Data Science \\
   1050 Brussels, Belgium}
\email{kevin.piterman@vub.be}


\begin{document}

\begin{abstract}
The purpose of this paper is to show that, under mild inductive assumptions, if a group $G$ contains a component that is a simple group of Lie type in characteristic $p$ and $O_p(G) = 1$, then the Quillen poset of $G$ at $p$ has nonzero rational homology. In particular, this shows that such components cannot arise in a minimal counterexample to Quillen's conjecture. This result is of particular interest at the prime $p=2$, where the conjecture is still open.
\end{abstract}

\maketitle

\section{Introduction}

Let $G$ be a finite group and $p$ a prime number.
Write $\Ap(G)$ for the poset of nontrivial elementary abelian $p$-subgroups of $G$ ordered by inclusion.
The study of collections of $p$-subgroups and the topological properties of their corresponding order complexes plays a role in several areas of mathematics, including representation theory, group cohomology, and the classification of finite simple groups (CFSG).
We refer to \cite{Smith} for more detailed discussions and applications.

The poset $\Ap(G)$ was introduced in \cite{Qui78} by Daniel Quillen, who focused on its homotopical properties in connection with intrinsic algebraic properties of $G$.
For instance, he proved that if $O_p(G) \neq 1$ then $\Ap(G)$ is contractible, and conjectured that the converse should hold as well. Namely, if $O_p(G) = 1$, then $\Ap(G)$ is not contractible.
In showing that $\Ap(G)$ is not contractible, one usually exhibits nonzero rational homology.
Quillen indeed established this stronger conclusion for solvable groups, groups of Lie type in characteristic $p$, and groups with $p$-rank at most two.
Therefore, in this article, we will work with such a version of the conjecture:

\medskip

\begin{center}
(H-QC) \quad If \ $O_p(G)=1$ \ then \ $\widetilde{H}_*(\Ap(G), \QQ) \neq 0$.
\end{center}

\medskip
In \cite{AS93}, Michael Aschbacher and Stephen D. Smith showed that (H-QC) holds for $p>5$ if $p$-extensions $LB$ of simple unitary groups $L\cong \PSU_n(q)$, with $p\mid q+1$ and $q$ odd, have nonzero homology in the largest possible degree:
\[ \widetilde{H}_{\dim \Ap(LB)}(\Ap(LB),\QQ)\neq 0. \]
That is, they satisfy the so-called Quillen dimension property $\QD_p$. Here, a $p$-extension $LB$ is a semidirect product of $L$ by an elementary abelian $p$-group $B$ inducing outer automorphisms on $L$ (that is, $B$ embeds into $\Out(L)$).

The reductions in \cite{AS93,Pit21,PS24} and in the present paper share a common underlying strategy. Given a component $L$ of a potential counterexample $G$ and a suitable $p$-extension $LB\leq G$, one constructs a nonzero cycle in a $p$-subgroup complex associated with $LB C_G(LB)$. By choosing the cycle on the $LB$-part with suitable maximality properties, one then shows that the resulting cycle remains nonzero in an ambient complex associated with $G$. In \cite{AS93} this ambient complex is $\Ap(G)$ itself, whereas \cite{Pit21,PS24} may first replace $\Ap(G)$ by a homotopy-equivalent complex, sometimes after collapsing contractible pieces. We discuss these differences in more detail in the strategy of the proof below.

The main methodological difference is that much of the reduction in \cite{AS93} relies on the use of the CFSG to determine all possible $p$-extensions of simple groups for $p$ odd. In particular, a key ingredient is the possibility of choosing, under suitable maximality conditions, the complement $B$ of a $p$-extension $LB\leq G$ so that
$O_p(C_G(LB))=1$.
Such a choice can always be made except when $G$ contains a component isomorphic to $\Sz(32)$ for $p=5$, or to $\PSL_2(8)$ or $\PSU_3(8)$ for $p=3$. Consequently, the reduction of \cite{AS93} applies directly for $p>5$.

The methods developed in \cite{Pit21,PS24} avoid these restrictions at the level of the reduction itself: they do not depend on the CFSG or on the parity of $p$, and they allow one to deal with situations where $O_p(C_G(LB))\neq 1$ by modifying the relevant $p$-subgroup complexes. The CFSG is then used only afterwards, when analysing the possible components of a minimal counterexample, the $\QD_p$ properties of their $p$-extensions, or verifying related criteria such as the existence of Robinson subgroups. The exceptional case $\Sz(32)$ for $p=5$ was eliminated in \cite{Pit21}, and the more general reductions of \cite{PS24} extend the approach of \cite{AS93} to every odd prime $p$. In particular, these methods also provide reductions when $p=2$, a case not covered by the original analysis of \cite{AS93}.

Finally, Antonio D{\'i}az Ramos \cite{DR} proved that, with some exceptions when $q=2,8$, $p$-extensions of unitary groups $\PSU_n(q)$, with $p$ odd and $p\mid q+1$, satisfy $\QD_p$. Combined with \cite{AS93,Pit21,PS24}, this establishes (H-QC) for odd primes.

For $p=2$, the reductions of \cite{Pit21, Pit24, PS24} eliminate many possible components of a minimal counterexample, but leave, for instance, certain groups of Lie type in characteristic $2$. More precisely, the remaining cases include components
\[ A_n(4^a),\qquad D_n(4^a),\qquad E_6(4^a), \]
when $G$ contains $2$-extensions $LB$ in which $B$ involves graph automorphisms, or graph and field automorphisms of order $2$, and $O_2(C_G(LB))=1$; see Remark 6.9 in \cite{PS24}.

The main result of this paper removes this restriction.

\begin{mainthm}
\label{thm:main}
Let $G$ be a finite group and $p$ a prime.
Assume that $O_p(G) = 1$ and that $L$ is a component of $G$ that is a simple group of Lie type in characteristic $p$.
If $C_G(LB)$ satisfies (H-QC) whenever $LB \leq G$ is a $p$-extension of $L$ in which $B$ is trivial or generated by a field automorphism of order $p$, then $G$ satisfies (H-QC).
\end{mainthm}

Notice that Theorem \ref{thm:main} allows $p$-extensions $L\widetilde{B}\leq G$ in which $\widetilde{B}$ has $p$-rank two or contains graph automorphisms, and imposes no condition on $O_p(C_G(L\widetilde{B}))$; compare Theorem 6.4 in \cite{PS24} and Theorems 2.3 and 2.4 in \cite{AS93}.

Combining Theorem \ref{thm:main} with the reductions obtained in \cite{Pit21,Pit24,PS22,PS24} yields the following reduction of (H-QC) at $p=2$.

\begin{mainthm}
\label{thm:quillenstatus}
Let $G$ be a minimal counterexample to (H-QC) for $p = 2$.
Then $G$ contains a component $L$ isomorphic to one of the following classical groups in characteristic $r\geq 5$:
\[ \PSL_n(q) (n\geq 4), \, \PSU_n(q) (n\geq 4) \, \PSp_{2n}(q) (n\geq 3), \, \Omega_{2n+1}(q) (n\geq 2), \, \POm^{\pm}_{2n}(q) (n\geq 4).\]
Moreover, some $2$-extension $LB\leq G$ of $L$ fails $\QD_2$.
\end{mainthm}

Indeed, the previous reductions show that every component of a minimal counterexample admits a $2$-extension contained in $G$ that fails $\QD_2$, while components of many other types can already be eliminated. Theorem \ref{thm:main} now eliminates components of Lie type in characteristic $2$. On the other hand, as explained in \cite[Section 8]{PS24}, if every component of $G$ admits Robinson subgroups, that is, satisfies property $\mathcal{R}(2)$ from page 96 of \cite{PS24}, then $G$ satisfies (H-QC) for $p=2$; see Proposition 8.1 in \cite{PS24}, and compare with page 491 of \cite{AS93}. In \cite[Section 10]{PS24}, this criterion is shown to apply to alternating and sporadic components arising in a minimal counterexample, as well as to components that are groups of Lie type in characteristic $3$. Together with the remaining reductions from \cite{Pit21,Pit24,PS22,PS24}, this leaves precisely the classical groups listed in Theorem \ref{thm:quillenstatus}.

Thus, Quillen's conjecture for $p=2$ is reduced to establishing $\QD_2$ for $2$-extensions of classical groups in characteristic $r\geq 5$.

\subsection*{Strategy of the proof}
We now explain the main ideas behind Theorem \ref{thm:main}.
Suppose that $L$ and $G$ satisfy its hypotheses.
For the $p$-extensions relevant to the argument, the Quillen poset $\Ap(LB)$ will generally fail $\QD_p$: its homotopy is controlled by a simplicial complex $\Delta(L)_{\F_f}$ of dimension at most $l$, where $l$ is the Lie rank of $L$, while usually
\[ l\ll \dim\Ap(LB)=m_p(LB)-1. \]
See Definition \ref{def:extendedComplex} and Eq. (\ref{eq:defFf}).
Nevertheless, $\Delta(L)_{\F_f}$, whose dimension is $l$ or $l-1$, always has nonzero homology in its top degree.
Thus, the role played by a $\QD_p$-cycle in the classical argument is replaced here by a top-dimensional cycle in $\Delta(L)_{\F_f}$, a complex introduced in \cite{Pit26}.

The strategy is therefore to deform $\Ap(G)$ into a homotopy-equivalent simplicial complex containing $\Delta(L)_{\F_f}$, and then propagate nonzero homology from a maximal-dimensional cycle in this subcomplex to the whole complex.

In \cite{PS24}, we already showed that the building $\Delta(L)$ can be embedded into a suitably reduced complex homotopy equivalent to $\Ap(G)$. The difficulty is that, when $G$ contains elements inducing graph or graph-field automorphisms on $L$, the subcomplex used there for homology propagation need not have nonzero homology in its maximal dimension. Such automorphisms are attached to $\Delta(L)$ along top-dimensional links that are nevertheless homotopy equivalent to complexes of strictly smaller dimension. This inflates the dimension by one without producing homology in the new top degree.

The results of \cite{Pit26} show that these elements can instead be attached to the building through homotopy-equivalent links of the correct dimension, corresponding to fixed-point subcomplexes of the building. In this way graph automorphisms do not increase the relevant top dimension, whereas field automorphisms increase it by one and produce nonzero homology in the new top degree.

The main new point of the present paper is to realise this construction inside a homotopy-equivalent version of $\Ap(G)$, for a suitable $p$-extension of the component $L$. The homotopy equivalences developed in Section 3 implement the required internal collapses and effectively replace the copy of $\Ap(L)$ inside the relevant model of $\Ap(G)$ by the face poset of the building.

Although these constructions extend to considerably more general subgroup configurations, and even to arbitrary simplicial complexes, we restrict ourselves to the present setting to keep the exposition focused on Quillen's conjecture. The homotopy equivalences can also be made $N$-equivariant for a suitable subgroup $N$ of $G$ normalising $L$.

Section 2 recalls the necessary material on poset topology, group theory, and the topology of Quillen posets. Section 3 develops the required homotopy equivalences. Finally, Section 4 is devoted to the proof of Theorem \ref{thm:main}.

\subsection*{Acknowledgements}
The author thanks Antonio D{\'i}az Ramos for many helpful comments on a preliminary version of this article.
This work was supported by the FWO mandate 12K1223N.

\section{Preliminaries}
\label{sec:preliminaries}

We follow the notation and conventions of Section 2 of \cite{PS24} concerning finite groups and components.
We refer the reader to \cite{GLS98} for standard results on groups of Lie type.
For more details on the topological tools, we refer to \cite{Pit26}.

All groups, posets, and simplicial complexes considered in this article are finite.

\subsection*{Posets and simplicial complexes}
For a poset $X$, we write $\Delta(X)$ for its order complex, and regard $X$ as a topological space via the topology of the geometric realisation of its order complex.
For simplicity, we will write $\widetilde{H}_*(X)$ for the reduced homology groups of $X$ with coefficients over the rationals.
For $x\in X$ and $Y\subseteq X$, write $Y_{>x}$ for the subposet of elements $y\in Y$ with $y>x$. Similarly define $Y_{\geq x}$, $Y_{<x}$ and $Y_{\leq x}$.
We say that $Y$ is downward (resp. upward) closed if $X_{\leq y}\subseteq Y$ (resp. $X_{\geq y}\subseteq Y$) for all $y\in Y$.
Recall that if $f,g:X\to Y$ are order-preserving maps between posets such that $f(x)\leq g(x)$ for all $x\in X$, then $f,g$ give rise to homotopic maps between the geometric realisations.

Posets will usually be denoted by the letters $X,Y,Z,W$ (plus decorations), while we reserve capital Greek letters to denote simplicial complexes.

For a simplicial complex $\Delta$, a subcomplex $\Lambda$ and a simplex $\sigma \in \Delta$, we write
\[ \Lk_{\Lambda}(\sigma) := \Lk_{\Delta}(\sigma)\cap \Lambda = \{ \tau \in \Lambda \tq \tau\cup \sigma\in \Delta, \tau\cap\sigma = \emptyset\}\]
for the link of $\sigma$ in $\Lambda$.
In case the ambient complex could be ambiguous, we use the notation $\Lk_{\Delta}(\sigma)\cap \Lambda$.

The following is a restatement of Theorem 2.1 from \cite{Pit26}.

\begin{theorem}
\label{thm:pushout}
Let $\Delta$ be a $G$-simplicial complex and $\Lambda$ a full $G$-invariant subcomplex such that no two distinct vertices in $\Delta\setminus \Lambda$ form a simplex.
Then we have a pushout diagram
    \[ \pushout{v}{v\in \Delta\setminus \Lambda}{\Lambda}{\Delta}\]
\end{theorem}

For a poset $X$, we write $X\cup 1$ for the poset obtained by adjoining a unique minimal element $1$, which will usually denote the trivial group in our context.
That is, $1 < x$ for all $x\in X$.
If $X,Y$ are posets, the pre-join $X\ujoin Y$ is the poset whose underlying set is $(X\cup 1)\times (Y\cup 1) \setminus \{ (1,1)\}$, with ordering $(x,y)\leq (x',y')$ if $x\leq x'$ and $y\leq y'$.
The pre-join $X\ujoin Y$ has the homotopy type of the usual join of posets (see \cite[Section 4]{PS24}).
The order complex of the latter is just the join $\Delta(X) * \Delta(Y)$ of simplicial complexes, whose geometric realisation is isomorphic to the classical join of topological spaces. 

The main tool to prove homotopy equivalences is Quillen's fibre theorem:

\begin{theorem}
[Quillen's fibre theorem]
\label{thm:Quillen}
Let $X,Y$ be posets and $f:X\to Y$ an order-preserving map.
If for all $y\in Y$ the lower fibre $f^{-1}(Y_{\leq y})$ is contractible, then $f$ is a homotopy equivalence.
\end{theorem}

Finally, we will sometimes use the following lemma, which is an alternative version of Quillen's fibre theorem that will mostly be applied in the context of simplicial complexes when changing a link $A'$ for a subcomplex $A$ with the same homotopy type.

\begin{lemma}
[{Gluing lemma}]
\label{lm:gluing}
Suppose we have a commutative diagram of topological spaces with continuous maps
\[
\begin{tikzcd}
X \arrow[d] & A \arrow[l, "f"] \arrow[r] \arrow[d] & Y \arrow[d]\\
X' & A' \arrow[l, "f'"] \arrow[r] & Y'    
\end{tikzcd}
\]
such that the vertical arrows are homotopy equivalences and $f,f'$ are (closed) cofibrations.
Then the map induced on the pushouts $X\cup_A Y\to X'\cup_{A'}Y'$ is a homotopy equivalence.
\end{lemma}

\begin{proof}
See Lemma 2.1.3 of \cite{May2}.
\end{proof}

\subsection*{Group theory preliminaries}
For a finite group $G$, a component is a subnormal and quasisimple subgroup.
Recall that a group is quasisimple if it is perfect and the quotient by its centre is (nonabelian) simple.
By a simple group, we will always mean a nonabelian simple group.
Components have many interesting properties, and we recall a few of them next, which we will implicitly use in this paper.
Two different components of $G$ commute, and hence the product of all components of $G$ (the layer of $G$) is a central product of quasisimple groups.
When the Fitting subgroup $F(G)$ is trivial, components are simple.
In particular, if $L$ is a component of $G$ whose centre $Z(L)$ is a $p'$-group, and $g\in G$ centralises a nontrivial $p$-subgroup of $L$, then $g$ normalises $L$. Notice that $Z(L)$ is a $p'$-group for instance when $O_p(G)=1$.
The generalised Fitting subgroup $F^*(G)$ of $G$ is the (central) product of $F(G)$ and all the components of $G$.
As usual, for subgroups $H,K$ of $G$, we write $N_K(H)$, resp. $C_K(H)$, for the normaliser, resp. centraliser, of $H$ in $K$.

Denote by $\Out(L)$ the outer automorphism group of a group $L$.
We say that $x\in \Aut(L)$ induces an outer automorphism on $L$ if $x$ is not inner, i.e., $x\notin \Inn(L)$. A subgroup $F$ of $\Aut(L)$ induces outer automorphisms on $L$ if every nontrivial element of $F$ induces an outer automorphism on $L$.

By a finite simple group of Lie type in characteristic $p$ we mean a simple group of the form $L = O^{p'}(\GG_\sigma)$, where $\GG$ is a simple linear algebraic group defined over an algebraic closure of $\GF{p}$, $\sigma$ is a Steinberg endomorphism of $\GG$ and $\GG_{\sigma}$ denotes the $\sigma$-fixed points.
In a few cases, the derived subgroup of $O^{p'}(\GG_\sigma)$ is the \textit{correct} simple group, but we will not need to deal with them here (we refer the reader to \cite{PS24} for more details on those cases).
It is well known that, up to inner transformations, a Steinberg endomorphism
$\sigma$ is either of ordinary type, in which case it can be written as a
commuting product
\[ \sigma=\gamma\Frob^a, \]
where $a\geq 1$ and $\gamma$ is induced by a graph automorphism of the Dynkin diagram of
$\GG$, or else $\sigma$ is of Suzuki--Ree type. In the
latter case, it has the form
\[ \sigma=\theta\Frob^a,\]
where $a\geq 0$ and $\theta$ is an exceptional isogeny interchanging long and short root
subgroups and satisfying $\theta^2=\Frob$. Thus
\[ \sigma^2=\Frob^{2a+1}.\]
When $\gamma\neq 1$ in the ordinary case, we say that $L$ is twisted by a
diagram symmetry of order $|\gamma|$. The Suzuki--Ree groups will be treated
as a separate twisted case.

On the other hand, $\Frob$ induces an automorphism of $\GG_\sigma$, and hence of $L$, and we write $\Phi_L$ for the subgroup of $\Aut(L)$ generated by the automorphism induced by $\Frob$.
By a field automorphism of $L$ we mean an $\Aut(L)$-conjugate of a nontrivial element of $\Phi_L$.
Moreover, $\Phi_L$ embeds as a normal subgroup of the quotient of $\Aut(L)$ by the group of inner-diagonal automorphisms $\Inndiag(L)$, and we denote its image also by $\Phi_L$.
When $\sigma$ is twisted or of Suzuki--Ree type, $\Aut(L) / \Inndiag(L) = \Phi_L$.
We say that $B\leq \Aut(L)$ induces field automorphisms of order $p$ on $L$ if $B$ is generated by an order-$p$ field automorphism of $L$.
In that case, the image of $B$ in $\Aut(L) / \Inndiag(L)$ lies in $\Phi_L$.

According to Definition 2.5.13 in \cite{GLS98}, in the case that $\gamma\neq 1$ (so $\gamma$ has order $r = 2$ or $3$, and $L$ is not a Suzuki--Ree group), a field automorphism whose order is divisible by $r$ is called a graph automorphism. We will not follow such a convention here.

In the twisted case, if $\gamma$ has order $r=p$, there are order-$p$ elements $x\in \Aut(L)$ with nontrivial image in $\Phi_L$ that do not induce field automorphisms on $L$.
In the GLS terminology, these are still called graph automorphisms, and here one has $O_p(C_L(x)) \neq 1$ (see Proposition 4.9.2 in \cite{GLS98}).
We refer to Appendix B in \cite{Pit26} for a more elaborate discussion.

Recall that every group of Lie type $L$ in characteristic $p$ has a BN-pair in characteristic $p$, and hence a building which we denote by $\Delta(L)$.

The following trichotomy is the key group-theoretic input used throughout the paper.

\begin{lemma}
\label{lm:trichotomy}
Let $L$ be a simple group of Lie type in characteristic $p$, and let $x\in \Aut(L)$ be an element of order $p$ that does not induce an inner automorphism on $L$.
Write $\overline{x}$ for the image of $x$ in $\Aut(L)/\Inndiag(L)$.
Then exactly one of the following holds:
\begin{enumerate}
    \item $\overline{x}\in \Phi_L$, $x$ induces an order-$p$ field automorphism on $L$ and $O_p(C_L(x)) = 1$.
    \item $\overline{x}\in \Phi_L$ but $x$ does not induce a field automorphism on $L$. Moreover, $L$ is twisted by a diagram symmetry of order $p$, $F^*(C_L(x)) = O_p(C_L(x))$ is a nontrivial $p$-group, and $L\gen{x} = L\gen{y}$, for some element $y\in L\gen{x}$ inducing an order-$p$ field automorphism on $L$.
    \item $\overline{x} \notin \Phi_L$. In this case, $x$ induces a nontrivial automorphism of the type diagram of the building of $L$. Consequently, $x$ does not pointwise fix any maximal flag in the building of $L$.
\end{enumerate}
\end{lemma}

\begin{proof}
Suppose first that $\overline{x}\in\Phi_L$. If $x$ induces an order-$p$
field automorphism on $L$, then we are in case (1) (see Proposition 4.9.1 and 4.9.2 of \cite{GLS98}). Otherwise, Proposition
4.9.2 of \cite{GLS98} implies that $L$ is twisted by a diagram symmetry of order $p$
and that
\[ F^*(C_L(x))=O_p(C_L(x))>1. \]
Moreover, Proposition B.2(3) of \cite{PS24} gives an element
$y\in L\gen{x}$ of order $p$ inducing a field automorphism on $L$ and with $L\gen{y}=L\gen{x}$.
Thus case (2) holds.

It remains to consider the case $\overline{x}\notin\Phi_L$. By the structure
of the outer automorphism group modulo diagonal automorphisms, the image of
$x$ has a nontrivial graph component. Equivalently, $x$ induces a nontrivial
permutation of the types of the building of $L$. If $x$ pointwise fixes a
maximal flag, then it would fix each vertex of the corresponding chamber. Since
a chamber contains one vertex of each type, this would force $x$ to act
trivially on the type set, a contradiction. Hence $x$ does not pointwise fix
any maximal flag in the building of $L$.
\end{proof}

Recall that a $p$-extension of a group $L$ is a split extension $LB = L\rtimes B$ where $B$ is an elementary abelian $p$-group that induces outer automorphisms on $L$. In particular, $B\in \Ap(\Out(L))\cup \{1\}$.


Denote the $p$-rank of $G$ by $m_p(G)$, which is the maximal possible dimension of an elementary abelian $p$-subgroup of $G$.
Then $\Delta(\Ap(G))$ has dimension $m_p(G)-1$, and $G$ satisfies the Quillen dimension property, $\QD_p$ for short, if $\widetilde{H}_{m_p(G)-1}(\Ap(G))\neq 0$.

When $L$ is a group of Lie type in characteristic $p$, it is known that $\Ap(L)$ has the homotopy type of the building associated with $L$ (cf. \cite{Pit26, Qui78}).
The dimension of $\Delta(L)$ is in most cases strictly smaller than that of $\Ap(L)$, and this implies that $p$-extensions of $L$ generally fail $\QD_p$.
Nevertheless, the main results from \cite{Pit26} show that for any $p$-extension of $L$, $\Ap(LB)$ has the homotopy type of a wedge of spheres; these spheres may have different dimensions when $LB$ contains graph automorphisms.
If $l$ denotes the Lie rank of $L$, then $\Ap(LB)$ has nonzero homology in degree $l-1 = \dim \Delta(L)$ if $LB$ does not contain order-$p$ field automorphisms, and $l$ otherwise.
This is the largest degree in which reduced homology is nonzero, and it coincides with the dimension of the following simplicial complex.

\begin{definition}
\label{def:extendedComplex}
Let $H$ be a group acting simplicially on a simplicial complex $\Delta$.
Let $\F$ be a set of subgroups of $H$.
We denote by $\Delta_{\F}$ the simplicial complex containing $\Delta$ and $\F$ as a discrete set of vertices, and with simplices $\sigma \cup \{F\}$ for $\sigma\in \Delta$ and $F\in \F$ if $F$ fixes $\sigma$ pointwise.
\end{definition}

Let
\begin{equation}
\label{eq:defFf}
\begin{split}
\F_f & = \{ C\in \Ap(LB) \tq C\text{ has order $p$ and induces field automorphisms on }L\},\\
\F_g & = \{D\in \Ap(LB) \tq D \text{ has order $p$, $D\notin \F_f$ and }O_p(C_L(D)) = 1\}.
\end{split}
\end{equation}
Then $\Ap(LB)$ has the homotopy type of $\big(\Delta(L)_{\F_f} \big)_{\F_g}$, and this complex has nonzero homology in the largest possible degree, which is $l$ if $LB$ contains order-$p$ field automorphisms (i.e., $\F_f\neq\emptyset$), and $l-1$ otherwise.
Notice that, in the complex $\big(\Delta(L)_{\F_f} \big)_{\F_g}$, an element of $\F_f$ or $\F_g$ fixes a simplex of $\Delta(L)$ pointwise if it normalises each of its vertices, which are parabolic subgroups.
Moreover, an element $D\in \F_g$ stabilises $C\in \F_f$ if and only if $D$ and $C$ commute.
See Corollary 1.2 and Theorem 5.15 in \cite{Pit26}.

We write $\Sp(G)$ for the poset of all nontrivial $p$-subgroups of $G$, and $\Bp(G)$ for its Bouc poset, consisting of those $R\in \Sp(G)$ such that
\[ R = O_p(N_G(R)). \]
It is well known that $\Ap(G)\simeq \Sp(G) \simeq \Bp(G)$.
Moreover, for $L$ as above, its Bouc poset is exactly the poset of radical subgroups of $L$, which is isomorphic to the face poset of the building, that is, the poset of proper parabolic subgroups ordered by reverse inclusion (see \cite{Pit26, Qui78}).

The following definition encodes some of the shared properties of the classical posets $\Ap$, $\Sp$ and $\Bp$.

\begin{definition}
Let $G$ be a group and $\P \subseteq \Sp(G)$ a subposet.
Let $F\leq G$ be a $p$-subgroup, and define
\[ \N(\P, F) = \{ P\in \P \tq C_P(F)\neq 1 \}. \]
We say that $\P$ is $F$-good if $\P^F \hookrightarrow \N(\P,F)$ is a homotopy equivalence.
\end{definition}

\begin{lemma}
\label{lm:downwardpGood}
Let $G$ be a group, $F\leq G$ a $p$-subgroup, and $\P \subseteq \Sp(G)$ a downward-closed subposet.
Then $\P$ is $F$-good.

In particular, if $H\leq G$, $\Ap(H)$ and $\Sp(H)$ are $F$-good for any $p$-subgroup $F\leq G$. Moreover, if $F\leq N_G(H)$ then also $\Bp(H)$ is $F$-good.
\end{lemma}

\begin{proof}
First, notice that if $P\in \P^F$, then $C_P(F)\neq 1$ by $p$-group actions.
As $\P$ is downward closed, if $P\in \N(\P, F)$, then $C_P(F)\in \P^F$.
Let $r:\N(\P, F) \to \P^F$ be defined by $r(P) = C_P(F)$, and write $i:\P^F\hookrightarrow \N(\P,F)$ for the natural inclusion.
Both maps are order preserving, $ir(P) = C_P(F)\leq P$, and $ri(P) = C_P(F) \leq P$, so $i,r$ are homotopy equivalences.

Since $\Ap(H)$  and $\Sp(H)$ are downward-closed subposets, the result immediately applies to these cases.

Finally, we show that $\Bp(H)$ is $F$-good when $F$ normalises $H$.
Recall that $\Bp(H)\hookrightarrow \Sp(H)$ is an equivariant homotopy equivalence (see for instance Proposition 3.1 in \cite{Pit26}).
We have the following diagram of inclusions.
\[\begin{tikzcd}[ampersand replacement = \&]
\Bp(H)^F \arrow[r, hook] \arrow[d, hook] \& \N(\Bp(H), F) \arrow[d, hook] \\
\Sp(H)^F	\arrow[r, hook] \& \N(\Sp(H), F)
\end{tikzcd}\]

Now, for $R \in \N(\Sp(H), F) \setminus \N(\Bp(H), F)$, we have
\[ \N(\Sp(H), F)_{>R} = \Sp(H)_{>R} \simeq \Sp(N_H(R)),\]
and since $R$ is not a radical $p$-subgroup of $H$, this upper interval is contractible.
By Quillen's fibre theorem, the inclusion
\[ \N(\Bp(H), F) \hookrightarrow \N(\Sp(H), F)\]
is a homotopy equivalence.
Hence every arrow in the diagram is a homotopy equivalence, and this finishes the proof. 
\end{proof}

\section{Homotopy equivalences}

The goal of this section is to replace the Quillen poset by a homotopy equivalent complex that only depends on the component $L$ and a suitable $p$-extension. We achieve this through a sequence of homotopy equivalences.

In what follows, we will work under the following setting:
\begin{equation}
    \label{eq:hyp}
    \begin{split}
      & \text{$L$ is a component of a group $G$ with $p\mid |L|$ and $p\nmid |Z(L)|$.}\\
    & \text{For a fixed $B\in \Ap(N_G(L))\cup\{1\}$ with $B\cap LC_G(L)=1$, we set $K:=C_G(LB).$}  
    \end{split}
\end{equation}

We will provide several deformations of the poset $\Ap(G)$ to construct a maximal-dimensional cycle for suitable $p$-extensions of $L$.

For arbitrary subgroups $H,N$ of $G$, we set
\begin{equation}
\label{def:Fposet}
\F_N(H) = \{ F\in \Ap(N) \tq F\cap H = 1 \}
\end{equation}
and give it the order induced by inclusion.

We will need the following distinguished subposet:

\begin{lemma}
\label{lm:conicalCentralizer}
Suppose that $G,L,B$ and $K$ are as in Eq. (\ref{eq:hyp}).
Then the following hold:
\begin{enumerate}
    \item The subposet
    \[ \{F\in \F_{N_G(L)}(LC_G(L)) \tq O_p(C_G(LF)) \neq 1\} \]
    is upward closed in $\F_{N_G(L)}(LC_G(L))$.
    \item If $B\neq 1$ then it induces outer automorphisms on $L$, so $LB$ is a $p$-extension of $L$.
    \item If $F\in \F_{N_G(L)}(LK)_{\geq B}$ then $F\in \F_{N_G(L)}(LC_G(L))$, that is, $F$ induces outer automorphisms on $L$.
    \item The subposet
    \begin{equation}
    \label{eq:defFB}
    \F_B := \{ F\in \F_{N_G(L)}(LK)_{\geq B} \tq O_p(C_G(LF)) \neq 1\}
    \end{equation}
    is upward closed in $\F_{N_G(L)}(LK)$. 
\end{enumerate}
\end{lemma}

\begin{proof}
Item (1) is \cite[Prop. 3.11]{PS24}.
Item (2) follows since $LC_G(L)$ is exactly the set of elements of $N_G(L)$ inducing inner automorphisms on $L$.
For item (3), if $F\in \F_{N_G(L)}(LK)$ with $F\geq B$, then $K \geq C_G(LF)$ and
\[ F\cap (LC_G(L)) = F\cap (LC_G(L)) \cap C_G(B) \leq F\cap LK = 1. \]
Hence $F\in \F_{N_G(L)}(LC_G(L))$ and $\F_B$ is upward closed in $\F_{N_G(L)}(LK)$ by item (1), so item (4) holds.
\end{proof}

Let $H,K$ be commuting subgroups of $G$ whose intersection is a $p'$-group.
Let $\P_H\subseteq \Sp(H)$ be a subposet.
Let $W_G(\P_H,K)$ be the poset whose underlying set is $\P_H\ujoin \Ap(K) \cup \F_G(HK)$ and the ordering relation $\preceq$ is given as follows.
We keep the set-theoretic containment ordering among members of $\F_G(HK)$, the natural ordering of $\P_H\ujoin \Ap(K)$, and for crossed terms $F\in \F_G(HK)$ and $(R,S)\in \P_H\ujoin \Ap(K)$ we set $F \prec (R,S)$ if $C_{RS}(F)\neq 1$.
It is proved in \cite[Theorem 4.10]{PS24} that indeed this defines a poset relation, and that
\[ W_G(\P_H,K)\simeq \Ap(G) \quad \text{if} \quad  \P_H = \Sp(H), \, \Bp(H) \text{ or } \Ap(H).\]
We will apply this construction in the special case that $H = L$ is a component of $G$ and $K = C_G(LB)$ for a certain $p$-extension $LB\leq G$.

We recall the poset constructed in \cite[Prop. 4.11]{PS24} for the case $B=1$ and $\P_L = \Sp(L), \Bp(L)$ or $\Ap(L)$.

\begin{definition}
\label{def:sqsubset}
Let $G,L,B$ and $K$ be as in Eq. (\ref{eq:hyp}), $\P_L\subseteq \Sp(L)$ a subposet, and let $\F_B$ be the poset from Eq. (\ref{eq:defFB}).
Define $\widetilde{W}_G(\P_L,B)$ as the poset which, as a set, is equal to $W_G(\P_L,K)$, but it has an ordering relation $\sqsubseteq$ given as follows.
We keep the inclusion ordering in $\F_G(LK)$ and in $\P_L\ujoin \Ap(K)$, and for crossed terms $F\in \F_G(LK)$ and $(R,S)\in \P_L \ujoin \Ap(K)$ we set 
$F\sqsubset (R,S)$ if and only if:
\begin{enumerate}[label=(\roman*)]
\item $C_{RS}(F)\neq 1$, and
\item $C_{RS}(F) \not\leq L$ if $F \in \F_B$.
\end{enumerate}
\end{definition}

Indeed, the proof of \cite[Prop. 4.11]{PS24} applies verbatim and shows that $\sqsubseteq$ defines a poset relation in $\widetilde{W}_G(\P_L, B)$.

We prove that we have a homotopy equivalence $W_G(\P_L,K) \simeq \widetilde{W}_G(\P_L, B)$.
This was established in \cite[Prop. 4.12]{PS24} in the case $B = 1$ and $\P_L$ one of the standard $p$-subgroup posets $\Ap(L), \Sp(L), \Bp(L)$.
We provide below a more general version of such a result.

\begin{theorem}
[First equivalence]
\label{thm:firstEquivalence}
Assume that $G,L,B$ and $K$ are as in Eq. (\ref{eq:hyp}).
Let $\P_L \subseteq  \Sp(L)$ be a poset of $p$-subgroups of $L$.
Then $W_G(\P_L,K) \simeq \widetilde{W}_G(\P_L,B)$.
\end{theorem}

\begin{proof}
Let $X = \Delta(\widetilde{W})$ and $Y = \Delta(W)$,
with $\widetilde{W} = \widetilde{W}_G(\P_L,B)$ and $W = W_G(\P_L,K)$.
It is clear that $X$ is a subcomplex of $Y$ since we have removed some simplices (but $\widetilde{W}$ is not a subposet of $W$).
We show that $X\hookrightarrow Y$ is a homotopy equivalence.

Let $X_i$ be the full subcomplex of $X$ spanned by the vertices in $\P_L\ujoin \Ap(K)$ and those $F\in \F_G(LK)$ such that $m_p(F) \leq i$.
Define $Y_i$ analogously.
Clearly, $X_i$ is a full subcomplex of $X_{i+1}$, and the vertices of $X_{i+1}$ not contained in $X_i$ consist of those $F\in \F_G(LK)$ with $m_p(F) = i+1$ (so two different ones do not form an edge).
Moreover, $X_0 = Y_0 = \Delta( \P_L \ujoin \Ap(K))$.
We show that, for all $i$, $X_i \hookrightarrow Y_i$ is a homotopy equivalence.

For that, we use the pushout diagrams as in Theorem \ref{thm:pushout} and the gluing Lemma \ref{lm:gluing}.
In fact, by induction, we only need to show that if $m_p(F) = i+1$, then $\Lk_{X_i}(F) \hookrightarrow \Lk_{Y_i}(F)$ is a homotopy equivalence.

Let $F\in \F_G(LK)$ be such that $m_p(F) = i+1$.
Let $\Q_F = (\P_L\ujoin \Ap(K))_{\sqsupset F}$, and $\U_F = \Ap(G)_{<F} = \Ap(F) \setminus \{F\}$.
Similarly, let $\R_F = (\P_L\ujoin \Ap(K))_{\succ F}$.
Since $W$ and $\widetilde{W}$ are posets, we have
\begin{equation}
\label{eq:linkDecompositionFirstEquivalence}
\Lk_{X_i}(F) = \Delta(\U_F * \Q_F) \ \text{ and } \ \Lk_{Y_i}(F) = \Delta(\U_F * \R_F).
\end{equation}

Thus, it is enough to show that $\Q_F \hookrightarrow \R_F$ is a homotopy equivalence.
Note that $\Q_F = \R_F$ if $F\notin \F_B$, so we consider next the case $F\in \F_B$.
We prove that $\Q_F$ and $\R_F$ are both contractible (and hence the inclusion is a homotopy equivalence).

First, notice that since $F$ normalises $L$ and contains $B$, $F$ normalises $K = C_G(LB)$.
Then $C_{RS}(F) = C_R(F)C_S(F)$ for $(R,S)\in \P_L\ujoin \Ap(K)$.

Suppose that $(R,S) \in \Q_F$.
It follows that $C_S(F)\neq 1$ since $C_{RS}(F)\nleq L$, and, in particular, $S\neq 1$.
Moreover,
\[ C_K(F) = C_G(LB) \cap C_G(F) = C_G(LF).\]
Hence, $(R,S) \in \Q_F \mapsto C_S(F)\in \Ap(C_G(LF))$ is a homotopy equivalence with homotopy inverse given by $C\mapsto (1,C)$.
This shows that $\Q_F \simeq \Ap(C_G(LF)) \simeq *$.

Now suppose that $(R,S) \in \R_F$.
We define a map:
\[ \varphi(R,S) = ( \phi_1(R), \phi_2(S) ) \in \N(\P_L,F) \ujoin \N( \Ap(K), F)\]
where
\[ \phi_i(A) = \begin{cases}
A & C_A(F) \neq 1,\\
1 & C_A(F) = 1.
\end{cases}\]
Clearly, $\phi_1,\phi_2$ are order-preserving maps.
Note also that $\phi_1(R),\phi_2(S)$ cannot be trivial at the same time.
Hence, $\varphi$ is a well-defined order-preserving map of posets.

There is also an inclusion $\psi :  \N(\P_L,F) \ujoin \N(\Ap(K), F) \hookrightarrow \R_F$ given by $\psi(R,S) = (R,S)$.
It is straightforward to check that this map is well-defined, i.e., its image lies in $\R_F$.
Now, note that $\psi \varphi(R,S)  \leq (R,S)$ and $\varphi \psi(R,S) = (R,S)$.
Thus $\varphi$ is a homotopy equivalence.
By Lemma \ref{lm:downwardpGood}, $\N(\Ap(K), F) \simeq \Ap(C_K(F)) = \Ap(C_G(LF))$, and the latter is contractible since $F\in \F_B$.
Hence, we see that $\N(\P_L,F) \ujoin \N(\Ap(K), F)$ is contractible as well.
This concludes the proof.
\end{proof}

Now we consider a different version of the ordering where we only "attach" elements of $\F_G(LK)$ to elements of $\P_L$ if they normalise it.

We proceed with a similar proof to show that this more restrictive condition does not change the homotopy type.

\begin{definition}
\label{def:tildaDelta}
Let $G$, $L$, $B$ and $K$ be as in Eq. (\ref{eq:hyp}), and 
let $\P_L\subseteq\Sp(L)$.
For $F\in \F_G(LK)$ and $(R,S)\in \P_L\ujoin \Ap(K)$, define the relation $F \dashv (R,S)$ if
\begin{enumerate}
    \item $C_{RS}(F)\neq 1$, and
    \item if $C_{RS}(F)\leq L$ then $F$ normalises $R$.
\end{enumerate}
Notice that if $1\neq C_{RS}(F)\leq L$ then automatically $F$ normalises $L$.

Define $\widetilde{\Delta}_G(\P_L, B)$ as the subcomplex of $\Delta(\widetilde{W}_G(\P_L,B))$ whose simplices are given by those $\sigma \in \Delta(\widetilde{W}_G(\P_L,B))$ such that
if $F\in \sigma\cap \F_G(LK)$ and $(R,S)\in \sigma$ then $F \dashv (R,S)$.
\end{definition}

Note that $\widetilde{\Delta}_G(\P_L, B)$ is still a flag complex, namely, a simplex is uniquely determined by its edges.

The above definition can also be extended to more general situations by replacing $L$ and $K$ by arbitrary commuting subgroups whose intersection has order prime to $p$.
For this paper, we avoid further generalisations and stick to our usual configuration involving $L, B$ and $K$.

\begin{theorem}
[Second equivalence]
\label{thm:secondEquivalence}
Let $G$, $L$, $B$ and $K$ be as in Eq. (\ref{eq:hyp}), and 
let $\P_L\subseteq\Sp(L)$ be a subposet that is $F$-good for every $p$-subgroup $F\leq N_G(L)$.
Then we have a homotopy equivalence
\begin{equation}
\label{eq:secondEquivalence}
\widetilde{\Delta}_G(\P_L,B) \hookrightarrow \Delta(\widetilde{W}_G(\P_L,B)).
\end{equation}
In particular, this applies to $\P_L = \Ap(L), \Bp(L)$ and $\Sp(L)$ and gives $\widetilde{\Delta}_G(\P_L,B) \simeq \Ap(G)$.
\end{theorem}

\begin{proof}
Let $Z = \widetilde{\Delta}_G(\P_L,B)$ and $X = \Delta(\widetilde{W}_G(\P_L,B))$.
Clearly we have an inclusion of complexes $Z\subseteq X$.
Let $X_i$ be the full subcomplex of $X$ whose set of vertices is the union of $\P_L \ujoin \Ap(K)$ and the subposet of $\F_G(LK)$ consisting of elements of $p$-rank at most $i$.
Let $Z_i = Z\cap X_i$.
We show that for all $i\geq 0$, $Z_i\hookrightarrow X_i$ is a homotopy equivalence.
Note that this is obviously true for $i = 0$ since
\[Z_0 = \Delta(\P_L \ujoin \Ap(K)) = X_0.\]
In view of the gluing Lemma \ref{lm:gluing}, it is enough to show that if $F\in \F_G(LK)$ and $m_p(F) = i+1$, then $\Lk_{Z}(F)\cap Z_i \hookrightarrow \Lk_{X}(F)\cap X_i$ is a homotopy equivalence.

Let us describe the link in each case.
Let $\U_F = \Ap(F) \setminus \{F\}$,
\[ \D_F = \{ (R,S)\in \P_L\ujoin \Ap(K) \tq F \dashv (R,S) \text{ and } F\sqsubset (R,S)\}, \]
and $\Q_F = (\P_L\ujoin \Ap(K))_{\sqsupset F}$.
By Eq. (\ref{eq:linkDecompositionFirstEquivalence}) in the proof of Theorem \ref{thm:firstEquivalence}, we know that
\[\Lk_{X_i}(F) = \Delta(\U_F * \Q_F) = \Delta(\U_F) * \Delta(\Q_F). \]

We claim that we have a similar decomposition now for $Z_i$.

\bigskip
\noindent
\textbf{Claim.} $\Lk_{Z_i}(F) = \Delta(\U_F) * \Delta(\D_F)$

\begin{proof}
Since $Z$ is a flag complex, it is enough to show that if 
$(R,S)\in \D_F$ and $A\in \U_F$, then $A \dashv (R,S)$ and $A\sqsubset (R,S)$.
Clearly $C_{RS}(A) \geq C_{RS}(F) \neq 1$.
Assume first that $C_{RS}(A) \leq L$.
Then $C_{RS}(F) \leq L$ and $F$ normalises $L$ and $R$ by definition of $\dashv$.
Thus $A$ normalises $L$ and $R$, so $A\dashv (R,S)$.
Now, if $A\in \F_B$, then $A\in \F_{N_G(L)}(LC_G(L))$ and $F\in \F_B$ by Lemma \ref{lm:conicalCentralizer}.
This contradicts the fact that $F$ and $(R,S)$ span a simplex with $C_{RS}(F)\leq L$, that is, $F\sqsubset (R,S)$.
Thus $A\notin \F_B$ and $A\sqsubset (R,S)$.

Next, if $C_{RS}(A)\nleq L$, then we immediately get $A\dashv (R,S)$ and $A\sqsubset (R,S)$.
\end{proof}

Finally, we prove that $\D_F\hookrightarrow \Q_F$ is a homotopy equivalence.
Suppose that $\D_F \neq \Q_F$.
Then there exists a pair $(R_0,S_0)\in \Q_F\setminus \D_F$ with $1\neq C_{R_0S_0}(F) \leq L$ and $F$ does not normalise $R_0$.
Notice that $F$ normalises $L$, and hence also $C_G(L)$.
In particular, we have $C_{RS}(F) = C_R(F)C_S(F)$ for any pair $(R,S) \in \P_L\ujoin \Ap(K)$.
Moreover, since $F \sqsubset (R_0,S_0)$, we see that $F\notin \F_B$.
This shows that $\Q_F$ is the poset $\R_F = (\P_L\ujoin \Ap(K))_{\succ F}$ from the proof of Theorem \ref{thm:firstEquivalence}.
Thus, we have the following inclusions of posets and homotopy equivalences: 
\begin{equation}
\label{eq:commutingDFFQF}
 \begin{tikzcd}
\P_L^F \ujoin \N(\Ap(K),F) \arrow[r, hook, "\simeq"'] \arrow[dr, hook, "i"] &  \N(\P_L,F)\ujoin \N(\Ap(K),F)  \arrow[r, hook, "\simeq"'] & \Q_F \\
& \D_F \arrow[ru, hook]
\end{tikzcd}
\end{equation}
Indeed, the first inclusion in the top row is an equivalence since $\P_L$ is $F$-good, and the second one follows from the proof of Theorem \ref{thm:firstEquivalence} by observing that we only need the decomposition $C_{RS}(F) = C_R(F)C_S(F)$ in that part (and the extra hypothesis there $B\leq F$ is not necessary for this equivalence).
Also, it is clear that $\P_L^F \ujoin \N(\Ap(K),F)$ is included in $\D_F$.
Hence, to show that $\D_F \hookrightarrow \Q_F$ is a homotopy equivalence, it is enough to prove that $i$ is a homotopy equivalence.
For that, we use Quillen's fibre theorem and prove that lower fibres are contractible.

Let $(R,S)\in \D_F$.
Then
\begin{align*}
i^{-1}\big( (\D_F)_{\leq (R,S)} \big) & = \{ (T,U) \in \P_L^F \ujoin \N(\Ap(K),F) \tq T\leq R, U\in \Ap(S)\cup\{1\}\}\\
& = {(\P_L^F)}_{\leq R} \ujoin \N(\Ap(S),F).
\end{align*}
Now, if $R\in \P_L^F$ then ${(\P_L^F)}_{\leq R}$ is contractible since it has a unique maximal element $R$, and hence so is the pre-join.
Otherwise, $F$ does not fix $R$, and since $1\neq C_{RS}(F) = C_R(F)C_S(F)$ with $F\dashv (R,S)$, we must have $C_S(F)\neq 1$.
This implies that $S \in \N(\Ap(S),F)$, so this poset is contractible because $S$ is its unique maximal element, and so the pre-join is contractible as well.
In any case, the preimage $i^{-1}\big( {\D_F}_{\leq (R,S)} \big)$ is contractible, and by Quillen's fibre theorem we conclude that $i$ is a homotopy equivalence.
By Eq. (\ref{eq:commutingDFFQF}), we conclude that the inclusion $\D_F\hookrightarrow \Q_F$ is a homotopy equivalence.
This shows that the inclusion in Eq. (\ref{eq:secondEquivalence}) is a homotopy equivalence.

The "In particular" assertion follows from Lemma \ref{lm:downwardpGood}.
\end{proof}

\section{Proof of Theorem \ref{thm:main}}

We now combine the homotopy equivalences of the previous section with a homology propagation theorem from \cite{PS24} to prove Theorem \ref{thm:main}.

\subsection*{Step 1: the setup}
Assume the hypotheses of Theorem \ref{thm:main}.
Thus, we have a group $G$ with $O_p(G) = 1$ and a component $L$ that is a simple group of Lie type in characteristic $p$.
Moreover, $C_G(LB)$ satisfies (H-QC) whenever $LB\leq G$ is a $p$-extension of $L$ with $B$ trivial or generated by a field automorphism of order $p$.
By Theorems \ref{thm:firstEquivalence} and \ref{thm:secondEquivalence}, we have
\[ \Ap(G) \simeq \widetilde{W}_G(\P_L,B) \simeq \widetilde{\Delta}_G(\P_L,B),\]
for $\P_L = \Bp(L)$ and $B \in \F_{N_G(L)}(LC_G(L)) \cup \{1\}$.
We will use the simplicial structure of $\widetilde{\Delta}_G(\P_L,B)$ to show that this complex has nontrivial homology.

Choose $B\in \F_{N_G(L)}(LC_G(L))\cup \{1\}$, with $B$ either trivial or generated by a field automorphism of order $p$, maximal subject to $O_p(C_G(LB)) = 1$.
Such a choice is possible since $O_p(C_G(L)) = 1$ (see for instance \cite[Lemma 2.3]{PS24}).
Consider the complex $\widetilde{\Delta} = \widetilde{\Delta}_G(\P_L, B)$ from Definition \ref{def:tildaDelta}, with $\P_L = \Bp(L)$.
Set $K = C_G(LB)$.

Let $O = \{B_0\in \Ap(LB) \tq B_0\cap L = 1\}$.
By Lemma \ref{lm:trichotomy}, we have that
\[ \{B_0\in O\tq O_p(C_L(B_0)) = 1\}\]
is exactly the poset $\F_f$ from Eq. (\ref{eq:defFf}).
In particular, if $B\neq 1$ then $B\in \F_f$ and two elements of $\F_f$ are $\Aut(L)$-conjugate.

\subsection*{Step 2: the homology propagation theorem}

We recall below \cite[Theorem 5.10]{PS24}, which is the homology propagation result that we will apply to our setting to show that $\widetilde{\Delta}$ has nontrivial homology. We use $\leq$ for an arbitrary poset ordering.

\begin{theorem}
\label{thm:homologyPropagation}
Let $X,Y,Z,W$ be finite posets such that:
\begin{enumerate}[label=(\roman*)]
    \item $(X\ujoin Y) \cup Z\subseteq W$;
    \item if $z\in Z$ is comparable with an element $w\in X\ujoin Y$, then $z < w$;
    \item if $z\in Z$, $(x,y)\in X\ujoin Y$ with $y\neq 1$, then $z < (x,y)$.
\end{enumerate}
Let $\Gamma$ be a subcomplex of $\Delta(W)$ containing $\Delta(X\ujoin Y \cup Z)$.
Suppose further that:
\begin{enumerate}
    \item $\Delta(X\times 1 \cup Z)$ has a nonzero homology cycle $\alpha$ in degree $m\geq 0$.
    \item There is a simplex $a_0$ occurring with nonzero coefficient in $\alpha$ such that if $v\in \Lk_{\Delta(W)}(a_0)\cap \Gamma$ is a vertex, then $v = (\max(a_0), y)$ for some $y\in Y$.
    That is, $\Lk_{\Delta(W)}(a_0)\cap \Gamma \subseteq \Delta( \{ \max(a_0)\} \times Y)$.
    \item There is a nonzero homology cycle $\beta$ for $Y$ in degree $n\geq -1$.
\end{enumerate}
Then the shuffle-product $T_*(\alpha, \beta)$ is a nonzero homology cycle for $\Gamma$ in degree $m+n+1$.
\end{theorem}

Here, the shuffle product is defined in $\Delta(X\times 1 \cup Z) \times \Delta(Y)$ as follows.
Let $a\in \Delta(X\times 1 \cup Z)$ and $b \in \Delta(Y)$, and write $a_Z = a\cap Z$, $a_X = a\cap (X\times 1)$.
The simplices $a, a_Z, a_X$ and $b$ have a natural ordering (orientation) inherited from the underlying poset structure.
Then
\begin{equation}
\label{eq:shuffle}
T_*(a,b) = \sum_{\sigma} (-1)^{\sigma}( a_Z \cup (a_X\times b)_{\sigma}),
\end{equation}
where we keep the given orientation in $a_Z$, and $(a_X\times b)_{\sigma}$ is the shuffle product of the chains $a_X$ and $b$, and $\sigma$ is a shuffle permutation of the set $a_X\cup b$ that preserves the original order within the two subsets $a_X$ and $b$ (see \cite[p. 483]{AS93} or \cite[p. 332]{PS24} for more details).
The operator $T_*$ is then defined by linear extension.

We apply this theorem to $W = \widetilde{W}_G(\Bp(L),B)$, $X = \Bp(L)$, $Y = \Ap(K)$ and $Z = \F_f$.
These posets clearly satisfy hypotheses (i)--(iii) of Theorem \ref{thm:homologyPropagation}.
Let
\[ \Gamma = \widetilde{\Delta} \cup \Delta(X\ujoin Y \cup Z),\] which is a subcomplex of $\Delta(W)$ containing $\Delta(X\ujoin Y \cup Z)$.

Our choice of $B$ implies that $O_p(K) = 1$, and $K = C_G(LB)$ satisfies (H-QC) by the hypothesis of Theorem \ref{thm:main}.
Thus, $Y = \Ap(K)$ has a nonzero homology cycle $\beta$, and condition (3) of Theorem \ref{thm:homologyPropagation} holds.

In what follows, we verify the remaining conditions (1) and (2) of Theorem \ref{thm:homologyPropagation}.
The propagation theorem only guarantees a nontrivial homology cycle in $\Gamma$. The main point of the remainder of the proof is to choose $a_0$ and $\alpha$ so that the propagated cycle $T_*(\alpha,\beta)$ actually belongs to the smaller complex $\widetilde{\Delta}$.

\subsection*{Step 3: choosing the cycle \texorpdfstring{$\alpha$}{alpha} and the simplex \texorpdfstring{$a_0$}{a0}}

If we regard $\Delta(\Bp(L))$ as a subcomplex of $\widetilde{\Delta}$ via the inclusion $R\mapsto (R,1)$, the intersection of $\Lk_{\widetilde{\Delta}}(B_0)$ with $\Delta(\Bp(L))$ is just $\Delta(\Bp(L)^{B_0}) = \Delta(\Bp(L))^{B_0}$ by definition of the relations $\dashv$ and $\sqsubset$.
Therefore, the full subcomplex spanned by $\Bp(L)$ and $\F_f$ inside $\widetilde{\Delta}$ is exactly the simplicial complex $\Delta(\Bp(L))_{\F_f}$ from Definition \ref{def:extendedComplex}.

On the other hand, the isomorphism between $\Bp(L)$ and the face poset of the building $\Delta(L)$ is $\Aut(L)$-equivariant.
From this, it is not hard to conclude that
\[\Delta(\Bp(L))_{\F_f} \simeq \Delta(L)_{\F_f},\]
where $\Delta(L)$ is the building of $L$.
By Theorem 5.9 of \cite{Pit26}, it follows that $\Delta(\Bp(L))_{\F_f}$ is spherical and has nonzero homology in the largest possible degree.
Therefore, we can take a maximal simplex $a_0$ containing some $B_0\in \F_f$ when $B\neq 1$, and occurring with nonzero coefficient in a nonzero homology cycle $\alpha$.
That is,
\begin{equation}
\label{eq:a0simplex}
a_0 = \begin{cases}
    \{ B_0 \sqsubset (R_1,1) < \cdots < (R_l, 1)\} & B \neq 1,\\
    \{ (R_1,1) < \cdots < (R_l, 1)\} & B = 1,
\end{cases}
\end{equation}
where $l$ is the Lie rank of $L$ and $B_0$ normalises each $R_i$.
Since $B$ and $B_0$ are $\Aut(L)$-conjugate and $LB = LB_0$, we can assume without loss of generality that $B = B_0$.

With this choice of $a_0$ and $\alpha$, condition (1) of Theorem \ref{thm:homologyPropagation} is satisfied.

\subsection*{Step 4: verifying the link condition}

We show condition (2) of Theorem \ref{thm:homologyPropagation} for the complex $\Gamma$.

Let $v\in \Lk_{\Delta(W)}(a_0)\cap \Gamma$ be a vertex.
We must prove that $v = (\max(a_0),y)$ for some $y\in Y$.

Suppose first that $v \in X\ujoin Y \cup Z$.
By maximality of $a_0 \setminus \{B\}$ in $X = \Bp(L)$, either $v \in \F_f$, or $v = (\max(a_0),y)$ with $y\in Y$.
Since either $\F_f$ is empty or it contains $B$ and all its elements have order $p$, we see that $v\notin \F_f$.
Thus $v = (\max(a_0),y)$ for some $y\in Y$.

Now suppose that $v \notin X\ujoin Y \cup Z$.
That is, $v = F \in \F_G(LK) \setminus \F_f$.
We prove that this leads to a contradiction.
By Eq. (\ref{eq:a0simplex}), $F\dashv (R_i,1)$ and $F\sqsubset (R_i,1)$, so $F\notin \F_B$ and $F$ normalises $L$ and $R_i$, for all $1\leq i \leq l$.
That is, $F$ fixes a maximal flag of $\Delta(L)$.

On the other hand, $F > B$ by assumption.
Hence,
\[ F\cap (LC_G(L)) = F\cap (LC_G(L)) \cap C_G(B)\leq F\cap LC_G(LB) = 1.\]
This means that $F$ acts faithfully on $L$ and induces outer automorphisms.
Since $F$ fixes a maximal flag of the building, by Lemma \ref{lm:trichotomy}, $F$ must inject into the cyclic group $\Phi_L$, and hence $F$ has order $p$.
This implies that $B = 1$.

Now, $F\notin \F_B$ implies that $O_p(C_G(LF)) = 1$. By maximality of $B$ and Lemma \ref{lm:trichotomy}, this forces $O_p(C_L(F)) \neq 1$.
But again, Lemma \ref{lm:trichotomy} gives an element $y\in LF$ that induces an order-$p$ field automorphism on $L$ and $LF = L\gen{y}$.
This contradicts the maximality of $B$ since
\[ O_p(C_G(L\gen{y})) = O_p(C_G(LF)) = 1. \]

This final contradiction shows that no such $v$ can exist, and therefore, every vertex in $\Lk_{\Delta(W)}(a_0)\cap \Gamma$ has the form $(\max(a_0),y)$ for some $y\in Y$.
This shows that condition (2) from Theorem \ref{thm:homologyPropagation} also holds with our choices.

\subsection*{Final step: the cycle \texorpdfstring{$T_*(\alpha,\beta)$}{T(alpha,beta)} lies in \texorpdfstring{$\widetilde{\Delta}$}{\~ Delta}}

Since our choices fulfil all the conditions from Theorem \ref{thm:homologyPropagation}, we get a nonzero homology cycle $T_*(\alpha, \beta)$ for $\Gamma$.

Since $\widetilde{\Delta}$ is a subcomplex of $\Gamma$, we verify that every (oriented) simplex involved in $T_*(\alpha, \beta)$ lies in $\widetilde{\Delta}$.
Indeed, by linearity, it is enough to check this for elements of the form $T_*(a,b)$ where $a$ and $b$ are oriented simplices involved in $\alpha$ and $\beta$ respectively.

This is clear if $Z = \F_f = \emptyset$, as in that case $\Gamma = \widetilde{\Delta}$.
Suppose then that $\F_f\neq\emptyset$.
An oriented simplex $a$ involved in $\alpha$ has the form
\[a =  \{ C \sqsubset (U_1,1) < \cdots < (U_l,1) \}, \]
with $a_Z = \{ C\}$, $a_X = \{ (U_1,1) < \cdots < (U_l,1)\}$, and $C\dashv (U_i,1)$ for all $i$ (recall that $\alpha$ is a cycle of the complex $\Delta(\Bp(L))_{\F_f}$).
An oriented simplex $b$ involved in $\beta$ has the form
\[ b = \{ (1,S_0) < \cdots < (1,S_n) \}.\]
We allow the case $n=-1$, in which case $b = \emptyset$.

Let $\sigma$ be a shuffle of $a_X\cup b$.
From Eq. (\ref{eq:shuffle}), we must prove that $a_Z \cup (a_X\times b)_{\sigma} \in \widetilde{\Delta}$.
This simplex either has the form
\[ C \sqsubset (R_1,1) < \cdots < (R_i,1) < (R_i,S_0) < \cdots < (R_l, S_n),\]
for some $1\leq i \leq l$, or 
\[ C \sqsubset (1,S_0) < \cdots < (R_l, S_n)\]
with $n\geq 0$.
Thus, it is clear that $C \dashv w$ for any term $w = (R,S)\in (a_X\times b)_{\sigma}$.
Therefore $a_Z \cup (a_X\times b)_{\sigma} \in \widetilde{\Delta}$, so $T_*(\alpha, \beta)$ gives a nonzero homology cycle for $\widetilde{\Delta}\simeq \Ap(G)$.
This completes the proof of Theorem \ref{thm:main}.


\end{document}